\documentclass[11pt]{amsart}

\usepackage{amssymb}
\usepackage{amsmath}
\usepackage{amsthm}
\usepackage{amscd}
\usepackage{amsfonts}
\usepackage{ascmac}
\usepackage[usenames]{color}
\usepackage{graphics}

\newtheorem{theorem}{Theorem}[section]

\newtheorem{definition}[theorem]{Definition}

\newtheorem{lemma}[theorem]{Lemma}
\newtheorem{problem}[theorem]{Problem}
\newtheorem{proposition}[theorem]{Proposition}
\newtheorem{remark}[theorem]{Remark}

\renewcommand{\bar}{\overline}

\renewcommand{\phi}{\varphi}

\newcommand{\bfC}{{\mathbf C}}

\newcommand{\bfR}{{\mathbf R}}

\newcommand{\calM}{{\mathcal M}}

\newcommand{\barR}{\overline {R}}
\newcommand{\barGamma}{\overline{\Gamma}}

\newcommand{\mapright}[1]{\smash{\mathop{   \hbox to 0.7cm{\rightarrowfill}}
  \limits^{#1}}}

\newcommand{\Ric}{\operatorname{Ric}}

\newcommand{\Tr}{\operatorname{Tr}}
\newcommand{\ad}{\operatorname{ad}}
\newcommand{\tr}{\operatorname{tr}}

\title{Quantum moment map and obstructions to the existence of closed Fedosov star products}
\author{
Akito Futaki and Laurent La Fuente-Gravy
}
\address{Yau Mathematical Sciences Center, Tsinghua University, Haidian district, Beijing 100084, China}
\email{futaki@tsinghua.edu.cn}
 \address{Mathematics Research Unit, Universit\'e du Luxembourg, 
 MNO, 6, Avenue de la Fonte, L-4364 Esch-sur-Alzette, Luxembourg}
\email{laurent.lafuente@uni.lu}

\begin{document}

\begin{abstract}
It is shown that the normalized trace of Fedosov star product for quantum moment map
depends only on the path component in the cohomology class of the symplectic form and the cohomology class of the closed formal 2-form
required to define Fedosov connections (Theorem \ref{main thm}).
As an application we obtain a family of obstructions to the existence of closed Fedosov star products 
naturally attached to symplectic manifolds (Theorem \ref{main thm1}) and K\"ahler manifolds (Theorem \ref{main thm2}).
These obstructions are integral invariants depending only on the path component of the cohomology class of the symplectic form.
Restricted to compact K\"ahler manifolds we re-discover an obstruction found earlier in \cite{LLF2}.
\end{abstract}

\maketitle

\section{Introduction}
A \emph{star product} \cite{BFFLS} on a Poisson manifold $M$ of dimension $n=2m$ is an associative product $*$ on the space $C^{\infty}(M)[[\nu]]$ of formal power series in $\nu$ with coefficients in $C^{\infty}(M)$ such that if we write
$$f*g:=\sum_{r=0}^{\infty} \nu^r C_r(f,g)\ \textrm{ for } f,\, g\in C^{\infty}(M)$$
then 
\begin{enumerate}
\item the $C_r$'s are bidifferential $\nu$-linear operators,
\item $C_0(f,g)=fg$ and $C_1(f,g)-C_1(g,f)=\{f,g\}$,
\item the constant function $1$ is a unit for $*$ (i.e. $f*1=f=1*f$).
\end{enumerate}
Recall that a symplectic form $\omega$ is a closed nondegenerate $2$-form. It induces the Poisson bracket $\{f,g\}:=-\omega(X_f,X_g)$ for $f,\, g\in C^{\infty}(M)$ and vector field $X_f$ uniquely determined by $\imath(X_f)\omega=df$. 
Any star product $\ast$ on a symplectic manifold $(M,\omega)$ admits a unique normalized trace
$$ \Tr : C^\infty(M)[[\nu]] \to \bfR[\nu^{-1},\nu]]$$
satisfying
$$ \Tr([f,g]_\ast) = 0.$$
Here normalization means as follows. On a contractible Darboux chart $U$ we have an equivalence
$B : (C^\infty(U)[[\nu]],\ast) \to (C^\infty(U)[[\nu]],\ast_{\mathrm{Moyal}})$
of $\ast|_{C^\infty(U)[[\nu]]}$ with the Moyal star product $\ast_{\mathrm{Moyal}}$ satisfying
$$ Bf\ast_{\mathrm{Moyal}} Bg = B(f\ast g).$$
The normalization condition is 
\begin{equation}\label{normalized_tr}
\Tr(f) = \frac1{(2\pi\nu)^m} \int_M Bf\ \frac{\omega^m}{m!}.
\end{equation}
It is known that the trace of a star product can always be written as an $L^2$-pairing with an essentially unique formal function $\rho\in C^{\infty}(M)[\nu^{-1},\nu]]$, called the trace density.
A star product is said to be (strongly) closed if 
the integration functional is a trace, c.f. \cite{CFS}. 
Equivalently, it means that the trace density is a formal constant, i.e. $\rho \in \bfR[\nu^{-1},\nu]]$. If such a closed star product exists, it is possible to define its character \cite{CFS}, a cyclic cocycle in cyclic cohomology. 
See \cite{fed}, \cite{NT}, \cite{gr} for more on the trace and the trace density.

There are known constructions of star products \cite{DWL}, \cite{fed2}, \cite{OMY}, \cite{Kon}. 
In this paper we consider Fedosov star product constructed in \cite{fed2} on symplectic manifolds. 
The Fedosov star product is defined 
given a symplectic connection $\nabla$ and a 
closed formal $2$-form $\Omega \in \nu \Omega^2(M)[[\nu]]$, and thus we denote it by $\ast_{\nabla, \Omega}$. 
Here, a symplectic connection means 
a torsion free connection making $\omega$ parallel. It is known (\cite{NT}, \cite{del}, \cite{bertcagutt}) that any star product on a symplectic manifold
is equivalent to a Fedosov star product.

In this paper, we study closedness of Fedosov star products naturally attached to symplectic or K\"ahler manifolds. On a compact symplectic manifold, we fix the de Rham class $[\omega_0]$ of the symplectic form. We study the following problem:

\begin{problem}[Symplectic version] \label{problem:symplectic}
Can one find a pair $(\omega,\nabla)$ consisting of a symplectic form $\omega \in [\omega_0]$ and a symplectic connection $\nabla$ with respect to $\omega$ such that $\ast_{\nabla,0}$ is closed?
\end{problem}

\noindent This problem is motivated by the study of moment map geometry of the space of symplectic connections. As noticed in \cite{LLF}, since the trace density of $*_{\nabla,0}$ is given by
\begin{equation*}\label{density1}
(2\pi\nu)^m\rho^{\nabla,0}:=1-\frac{\nu^2}{24} \mu(\nabla) + O(\nu^3)
\end{equation*}
where $\mu(\nabla)$ is the Cahen-Gutt momentum \cite{cagutt} of the symplectic connection $\nabla$, an affirmative answer to the above problem for $*_{\nabla,0}$ implies the constancy of the Cahen-Gutt momentum $\mu(\nabla)$. Recall that $\mu(\nabla)$ is given by
\begin{equation*}\label{density2}
\mu(\nabla):=(\nabla^2_{(p,q)} \Ric^{\nabla})^{pq} - \frac{1}{2} \Ric^{\nabla}_{pq}\Ric^{\nabla\;pq}+\frac{1}{4}\mathrm{R}^{\nabla}_{pqrs}\mathrm{R}^{\nabla \,pqrs},
\end{equation*}
where $\mathrm{R}^{\nabla} $ is the curvature of $\nabla$ and $\Ric^{\nabla}(\cdot,\cdot):=\tr[V\mapsto \mathrm{R}^{\nabla}(V,\cdot)\cdot]$
is the Ricci tensor.

On a closed K\"ahler manifold $(M,\omega_0,J)$, one consider the space $\calM_{[\omega_0]}$ of K\"ahler forms in the cohomology class of $\omega_0$, the complex structure being fixed. To $\omega \in \calM_{[\omega_0]}$, one can attach a natural family of Fedosov star products $\ast_{\nabla,\Omega_k(\omega)}$ described as follows. For $k\in \bfR$, consider closed $2$-form
$$\Omega_k(\omega):=\nu\,k\Ric(\omega),$$
with $\Ric(\omega):=\Ric^{\nabla}(J\cdot,\cdot)$ being the Ricci form of the K\"ahler manifold $(M,\omega,J)$.

\begin{problem}[K\"ahler version] \label{problem:kahler}
For a fixed real number $k$, can one find $\omega\in \calM_{[\omega_0]}$ with Levi-Civita connection $\nabla$ and Ricci form $\Ric(\omega)$ such that $\ast_{\nabla,\Omega_k(\omega)}$ is closed?
\end{problem}

\noindent A trace density for $\ast_{\nabla,\Omega_k(\omega)}$ is given by
$$(2\pi\nu)^m\rho^{\nabla,\Omega_k(\omega)}=1-\frac{\nu\,k}{2} \,\mathrm{S}_{\omega} + O(\nu^2),$$
with $\mathrm{S}_{\omega}$ being the scalar curvature (see Remark \ref{remark:rho1kahler}). So a necessary condition for $\ast_{\nabla,\Omega_k(\omega)}$, with $k\neq 0$, to be closed is the existence of a constant scalar curvature K\"ahler metric. 

Our obstructions come from the presence of symmetries of the symplectic manifolds. When a compact Lie group $G$ acts on $(M,\omega)$, it is natural to restrict the above problem on $G$-invariant symplectic forms in $[\omega_0]$ and to consider $G$-invariant Fedosov star products (built with $G$-invariant $\nabla$ and $\Omega$). An important feature in this context is the notion of quantum moment map \cite{MullerNeumaier,Xu,gr3,MullerNeumaierWick} which leads to phase space reduction in deformation quantization \cite{BHW,fedreduction}.
Let $G$ be a compact Lie group acting effectively on a compact symplectic manifold $M$ preserving
the symplectic form $\omega$, a closed formal $2$-form 
$\Omega \in \nu \Omega^2(M)[[\nu]]$ and a symplectic connection $\nabla$ so that the Fedosov star product $*_{\nabla,\Omega}$ is G-invariant. 
We identify a Lie algebra element $X \in \mathfrak g$ with a vector field on $M$ by the action of $G$. In \cite{MullerNeumaier,Xu,gr3}, 
a map $\mu_{\cdot} : \mathfrak g \to C^{\infty}(M)[[\nu]]$ is called a quantum moment map
if $\mu_{\cdot}$ is a Lie algebra morphism with respect to the commutator $\frac{1}{\nu}[\cdot,\cdot]_{*_{\nabla,\Omega}}$ on $C^{\infty}(M)[[\nu]]$ satisfying
\begin{equation}\label{qmoment}
X(u) = \frac1\nu \ad_{\ast_{\nabla,\Omega}} \mu_X (u)
\end{equation}
for any $u\in C^{\infty}(M)[[\nu]]$.
It follows from  Theorem 8.2 in \cite{gr3} or Deduction 4.4 in \cite{MullerNeumaier} that \eqref{qmoment} is equivalent to 
\begin{equation}\label{qmoment3}
 d \mu_X = i(X)(\omega - \Omega).
\end{equation}

If a formal function $f \in C^\infty(M)[[\nu]]$ satisfies 
$i(X_f) (\omega - \Omega) = df$
for some vector field $X_f$ we call $X_f$ the quantum Hamiltonian vector field of $f$, and also call $f$ the quantum
Hamiltonian function of $X_f$. In this paper we adopt \eqref{qmoment3} as the definition of
quantum moment map without mentioning $\nabla$, 
and we say that, given a symplectic form $\omega$ and a closed 2-form $\Omega$, 
a $G$-equivariant map $\mu : M \to \mathfrak g^\ast[[\nu]]$
is a quantum moment map if $\mu_X := \langle \mu, X \rangle \in C^\infty(M)[[\nu]]$ is a quantum Hamiltonian function of 
$X \in \mathfrak g$.
If there is a quantum moment map 
we say that $G$-action on $(M, \omega, \Omega)$ is quantum-Hamiltonian. 
Naturally if $(M, \omega, \Omega)$ is quantum-Hamiltonian $G$-space then $\omega$ and $\Omega$ are $G$-invariant if $G$ is connected.
Given a $G$-invariant symplectic connection $\nabla$ and Fedosov star product $\ast_{\nabla,\Omega}$, the quantum moment map
$\mu : M \to \mathfrak g^\ast[[\nu]]$ in our sense induces $\mu_{\cdot} : \mathfrak g \to C^{\infty}(M)[[\nu]]$ satisfying \eqref{qmoment}
by Theorem 8.2 in \cite{gr3} or Deduction 4.4 in \cite{MullerNeumaier} as quoted above. 
Further, by the $G$-equivariance we required for $\mu : M \to \mathfrak g^\ast[[\nu]]$,
the induced map $\mu_{\cdot} : \mathfrak g \to C^{\infty}(M)[[\nu]]$ is a Lie algebra morphism and thus a quantum moment map in the sense of 
\cite{MullerNeumaier,Xu,gr3}. 

Quantum moment maps are not unique, and any two of them differ by a map $b:\mathfrak g \to \mathbb{R}[[\nu]]$ that vanishes on Lie bracket. As a consequence, we can assume the quantum moment map is normalized so that 
\begin{equation}\label{normalize_c}
 \int_M \mu_X (\omega - \Omega)^m = 0.
\end{equation}

Given a quantum-Hamiltonian $G$-space $(M, \omega_0, \Omega_0)$, we denote by $\mathcal C^G([\omega_0],[\Omega_0])$ 
the space consisting of all triples $(\omega, \Omega, \nabla)$
such that 
\begin{enumerate}
\item[(a)] $(M,\omega,\Omega)$ is a quantum-Hamiltonian $G$-space, 
\item[(b)] $\omega$ is cohomologous to $\omega_0$ and there is a smooth path $\{\omega_s\}_{0\le s \le 1}$ consisting of $G$-invariant
symplectic forms joining $\omega_0$ and $\omega$ in the cohomology class $[\omega_0]$, 
\item[(c)] $\Omega$ is cohomologous to $\Omega_0$, and
\item[(d)] $\nabla$ is a $G$-invariant symplectic connection with respect to $\omega$.
\end{enumerate}
For each triple $(\omega, \Omega, \nabla)$ in $\mathcal C^G([\omega_0],[\Omega_0])$ we have the Fedosov star product $\ast_{\nabla,\Omega}$.

\begin{theorem}\label{main thm}
Let $(M, \omega_0, \Omega_0)$ be a quantum-Hamiltonian $G$-space and consider a triple 
$(\omega, \Omega, \nabla)$ in 
$\mathcal C^G([\omega_0],[\Omega_0])$. For $X \in \mathfrak g$, let $\mu_X$ be the quantum Hamiltonian
function of $X$ with respect to $\omega - \Omega$ with normalization \eqref{normalize_c}. 
Then the trace $\Tr^{\ast_{\nabla,\Omega}}(\mu_X)$
of the Fedosov star product $\ast_{\nabla,\Omega}$ is independent of the choice of $(\omega, \Omega, \nabla)$ in 
$\mathcal C^G([\omega_0],[\Omega_0])$.
\end{theorem}

\noindent Hence, one can define a symplectic invariant :

\begin{definition}\label{invariant2}
We define a character $\Tr^{[\omega_0],[\Omega_0]} : \mathfrak g \to \bfR[[\nu]]$ by 
$$ \Tr^{[\omega_0],[\Omega_0]} (X) := \Tr^{\ast_{\nabla,\Omega}}(\mu_X)$$
where the right hand side is given by Theorem \ref{main thm} with normalization as in \eqref{normalize_c}. 
\end{definition}

In the particular case $\Omega=0$, we obtain an obstruction to the existence of closed Fedosov star products, answering to Problem \ref{problem:symplectic}.

\begin{theorem}\label{main thm1}
Let $(M,\omega_0)$ be a compact symplectic manifold. If there exists a closed Fedosov star product
$\ast_{ \nabla,0}$ for $(\omega, 0, \nabla)$
in $\mathcal C^G([\omega_0],0)$ then $\Tr^{[\omega_0],0}$ vanishes.
\end{theorem}

Expanding $\Tr^{[\omega_0],0}(X)$ in terms of power series in $\nu$ we obtain a series of integral invariants
obstructing the existence of closed Fedosov star products.
The $\nu^{2-m}$-term is exactly the invariant found in \cite{LLF2}. See also \cite{FO_CahenGutt} for a different derivation
of this invariant using Donaldson-Fujiki type picture. This is one of the obstructions to asymptotic Chow semi-stability 
found by the first author in \cite{Fut04}. 
As discussed in \cite{FF}, the trace density is considered to play the same role as the Bergman function for the Berezin-Toeplitz star product
\cite{Schlich}, \cite{BMS}. 
See also \cite{donaldson01}, \cite{Fut05}, \cite{FO_SugakuExp}, \cite{FOS}, \cite{Fut12}, \cite{DVZ10}, \cite{Yamashita}, \cite{Ioos} for related topics.

On a compact K\"ahler manifold $(M,\omega_0,J)$ admitting an effective action of a compact Lie group $G$ preserving $\omega_0$ and $J$, it is natural to study $\calM^G_{[\omega_0]}$ the space of $G$-invariant K\"ahler forms in the cohomology class of $\omega_0$. For $\omega \in \calM^G_{[\omega_0]}$ and $k\in \bfR$, the  closed $2$-form $\Omega_k(\omega)$ is $G$-invariant. Thus, for $\omega \in \calM^G_{[\omega_0]}$ and $k\in \bfR$, we may consider the $G$-invariant Fedosov star product $*_{\nabla,\Omega_{k}(\omega)}$, where $\nabla$ is the Levi-Civita connection of the K\"ahler form $\omega$. 

Assume $\omega \in \calM^G_{[\omega_0]}$ makes $(M,\omega,0)$ a quantum-Hamiltonian $G$-space with quantum moment map $\mu_{\cdot}$ normalized by \eqref{normalize_c}. 
 Then, we will show that the triple $(\omega,\Omega_k(\omega),\nabla)$ is in $\mathcal C^G([\omega_0],[\Omega_k(\omega_0)])$ with 
some quantum moment map, which we denote by $\mu^k_{\cdot}$, normalized by \eqref{normalize_c}, i.e. in this case
\begin{equation*}
 \int_M \mu^k_X (\omega - \Omega_k(\omega))^m = 0
\end{equation*} 
for any $X \in \mathfrak g$. 
Another natural normalization for quantum moment maps is given by the integral. We define $\tilde{\mu}^k_{\cdot}$ to be the quantum moment map of the quantum-Hamiltonian $G$-space $(M,\omega,\Omega_k(\omega))$ normalized by 
\begin{equation}\label{eq:normalizedkahler}
\int_M\tilde{\mu}^k_{\cdot}\omega^m=0.
\end{equation}
In Proposition \ref{prop:quantumkahler}, we show $\tilde{\mu}^k_{\cdot}$ differs from $\mu^k_{\cdot}$ by a K\"ahler invariant, i.e. a constant depending only on the
K\"ahler class. Applying Theorem \ref{main thm}, we obtain a K\"ahler invariant obstructing the closedness of the Fedosov star product $*_{\nabla,\Omega_k(\omega)}$.

\begin{theorem}\label{main thm2}
Let $(M,\omega_0,J)$ be a compact K\"ahler manifold with $(\omega,0,\nabla) \in \mathcal C^G([\omega_0],0)$. Then for all $k\in \bfR$, 
$$\Tr^{\calM^G_{[\omega_0]},k}(X):=\Tr^{*_{\nabla,\Omega_k(\omega)}}(\tilde{\mu}^k_X)$$
is independent of the choice of $\omega \in \calM^G_{[\omega_0]}$. 
Moreover, if there exists a closed Fedosov star product
$\ast_{\nabla,\Omega_k(\omega)}$ for $\omega\in \calM^G_{[\omega_0]}$, then $\Tr^{\calM^G_{[\omega_0]},k}(X)$ vanishes.
\end{theorem}

The plan after this introduction is as follows. In section 2, we review Fedosov's construction of star product, particularly Fedosov connection
on Weyl algebra bundle. The description of flat sections in Darboux charts is given in section 2. 
In section 3, the variation formula of the trace is given. In section 4, we apply the variation formula to the quantum moment map, discuss on the two normalizations
\eqref{normalize_c} and \eqref{eq:normalizedkahler} and prove Theorem \ref{main thm}, \ref{main thm1} and \ref{main thm2}. In section 5, we give explicit formulas of the invariants up to terms in $\nu^2$.

\section{Prelimaries}
In this section we describe the equivalence $B$ in \eqref{normalized_tr} when $\ast$ is the Fedosov star product $\ast_{\nabla,\Omega}$ (c.f. \cite{fed2}, \cite{fed}). We mainly follow Fedosov's paper \cite{fedtrace} incorporating non-zero $\Omega$.
We first recall the construction of Fedosov star product. Let $T_x M$ be the tangent space at $x \in M$ of the
symplectic manifold $M$ with symplectic form $\omega$. We choose a basis $(e_1, \cdots, e_n)$ of $T_x M$ 
and write $\omega_{ij} = \omega(e_i,e_j)$ and express a tangent vector $y$ as $y = y^1e_1 + \cdots y^n e_n$.
Typically, we may take $e_i = \partial/\partial x^i$ for a choice of local coordinates $(x^1, \cdots x^n)$. 
The formal Weyl algebra $W_x$ corresponding to the symplectic space $T_x M$ is an associative algebra
consisting of the formal series
\begin{equation}\label{Weyl alg}
 a(y,\nu) = \sum_{k, \ell \ge 0} \nu^k a_{k,i_1, \cdots, i_\ell} y^{i_1} \cdots y^{i_\ell}
 \end{equation}
where $\nu$ is a formal parameter and $a_{k,i_1, \cdots, i_\ell}$ are real coefficients.
The product $\circ$ of the elements $a, b \in W_x$ is defined by the {\it Weyl rule}
$$ (a\circ b)(y,\nu) =\left. \left[ \exp\left(\frac\nu2 \Lambda^{ij} \frac\partial{\partial y^i} \frac \partial{\partial z^j}\right) a(y,\nu) b(z,\nu)\right]\right|_{y=z}$$
where $(\Lambda^{ij})$ is the inverse matrix of the symplectic form $(\omega_{ij})$. Note that this 
description of $\circ$ is
independent of the choice of a basis of $T_x M$.
We prescribe the degree, called the {\it Weyl degree}, of the variables by $\deg y^i = 1$ and $\deg \nu = 2$. 
Then each term of \eqref{Weyl alg} has Weyl degree $2k + \ell$. 

Taking a union of $T_x M$ over all $x \in M$ we obtain a bundle $W$ of the formal Weyl algebras. 
The local 
sections of $W$ are of the form
$$ a(x,y,\nu) = \sum_{2k+\ell \ge 0} \nu^k a(x)_{k,i_1, \cdots, i_\ell} y^{i_1} \cdots y^{i_\ell}$$
where the Weyl degree is used in the summation expression. These can be regarded as sections
of $\sum_{r \ge 0} \nu^r \sum S^\ell T^\ast M$. The product $\circ$ extends to an algebra structure on the space $\Gamma(W)$ of the sections the Weyl algebra bundle $W$ by
$$ (a\circ b)(x,y,\nu) =\left. \left[ \exp\left(\frac\nu2 \Lambda^{ij} \frac\partial{\partial y^i} \frac \partial{\partial z^j}\right) a(x,y,\nu) b(x,z,\nu)\right]\right|_{y=z}$$

We set $\Gamma(W)\otimes \wedge(M)$ to be the set of the Weyl algebra bundle valued differential forms which are expressed locally as
$$ \sum_{2k+\ell \ge 0} \nu^k a(x)_{k\,i_1 \cdots i_\ell\,j_1\cdots j_p}\, y^{i_1} \cdots y^{i_\ell}\,dx^{j_1} \wedge
\cdots \wedge dx^{j_p}.$$
We extend $\circ$ to $\Gamma(W)\otimes \wedge(M)$ by
$$ a\otimes\alpha \circ b\otimes \beta = a\circ b \otimes \alpha \wedge \beta$$
where $a,\ b \in \Gamma(W)$ and $\alpha,\ \beta \in \wedge(M)$. Then the commutator is naturally described as
$$ [s,s^\prime] = s\circ s^\prime - (-1)^{q_1q_2} s^\prime\circ s $$
for $s \in \Gamma(W)\otimes \wedge^{q_1}(M)$ and $s^\prime \in \Gamma(W)\otimes \wedge^{q_2}(M)$. 
Note that the center of the algebra $\Gamma(W)\otimes \wedge(M)$ consists of the elements of the form
$\sum_{k=0}^\infty \nu^k \alpha_k$ with $\alpha_k$ differential forms in $\wedge(M)$. We call these elements
the {\it central elements} or {\it central forms}.

It is well-known that there is a torsion-free connection making $\omega$ parallel, called a {\it symplectic connection}. 
In terms of the Christoffel symbols the condition for symplectic connection is that 
$\omega_{i\ell}\Gamma^\ell_{jk}$ is symmetric in $i,\ j,\ k$.
It is not unique, and for any two symplectic connections with Christoffel symbols $\Gamma^i_{jk}$ and $\Gamma'{}^i_{jk}$, 
$\omega_{i\ell}(\Gamma^\ell_{jk} - \Gamma'{}^\ell_{jk})$ is symmetric in $i,\ j,\ k$. Conversely, given a symplectic
connection and a symmetric covariant 3-tensor one can construct another symplectic connection in this way. 
Thus the space
of symplectic connections on $(M,\omega)$ is an affine space modeled on the vector space of all symmetric covariant
tensors of degree 3. 

Let $\nabla$ be a symplectic connection on $(M, \omega)$, and $\Gamma_{ij}^k$ be its Christoffel symbols. 
Let $\Gamma(W)\otimes\wedge(M)$ be the space of $W$-valued differential forms on $M$.
Then the induced exterior covariant derivative $\partial$ on $\Gamma(W)\otimes\wedge(M)$ is described as
$$ \partial a := da + \frac1\nu [\overline{\Gamma},a]$$
where 
$$ \barGamma = \frac12 \omega_{\ell k} \Gamma^k_{ij}y^\ell y^j dx^i.$$
Its curvature is described as
$$ \partial^2 a = \frac1\nu [\barR, a] $$
where
$$ \overline{R} = \frac14 \omega_{ir} \mathrm{R}^r{}_{jk\ell}y^i y^j dx^k \wedge dx^\ell.$$
For a $W$-valued 1-form $\gamma \in \Gamma(W)\otimes \wedge^1(M)$ we consider a more general connection
$$ D = \partial + \frac1\nu [\gamma,\cdot]. $$
The connection $D$ is determined up to a central term of $\gamma$. For the uniqueness of $\gamma$ we require
\begin{equation*}\label{W-normalization}
\gamma_0 := \gamma|_{y = 0} = 0.
\end{equation*}
This condition is called the {\it Weyl normalization}. The curvature $\Theta$ of $D$ is given by
\begin{equation*}\label{W-curvature}
\Theta = \overline{R} + \partial\gamma + \frac1\nu \gamma\circ\gamma.
\end{equation*}
Following \cite{fed2} we call $\Theta$ the {\it Weyl curvature} when $D$ satisfies the Weyl normalization.

We wish to obtain a flattening $D = \partial + \frac1\nu[\gamma,\cdot]$ of $\partial$ in the form
$$ \gamma = \omega_{ij}y^idx^j  + r$$ 
for some $r \in \Gamma(W)\otimes \wedge^1(M)$.
Since 
$$ \delta = dx^\ell\wedge \frac{\partial}{\partial y^\ell} = -\frac1\nu [\omega_{ij}y^idx^j, \cdot]$$
we may put
\begin{equation}\label{r}
D := \partial - \delta + \frac 1\nu [r,\cdot]
\end{equation}
and seek $r$ such that $D^2 = 0$. In \cite{fed2} such $D$ is called an {\it Abelian connection}.
Under the Weyl normalization condition $r|_{y=0} = 0$, one can see using $\delta^2 = 0$ and $\delta \partial + \partial \delta = 0$ that the Weyl curvature $\Theta$ is given by
$$ \Theta = - \frac12 \omega_{ij} dx^i \wedge dx^j + \overline{R} - \delta r + \partial r + \frac1\nu r\circ r.$$
Since 
$$ D^2 = \frac1\nu [\Theta, \cdot]$$
$D$ is Abelian if $\barR + \partial r - \delta r + \frac1\nu r\circ r$ is a central
2-form, that is, a 2-form in $\nu \Omega^2(M)[[\nu]]$. Introduce an operator $\delta^{-1}$ by
\begin{equation*}
\delta^{-1}(a_{pq}) = 
\left \{
\begin{array}{l}
\frac1{p+q}\, y^k \,i(\frac{\partial}{\partial x^k})a_{pq} \qquad (p+q \ne 0) \\
0 \qquad\qquad\qquad\qquad\ \, (p+q = 0)
\end{array}
\right.
\end{equation*}
for $a_{pq} \in \Gamma(W)\otimes \wedge^q(M)$ with degree $p$ symmetric term in $y$.
By the Hodge decomposition (see e.g. (5.1.7) in \cite{fed2}), we have for $a \in \Gamma(W)\otimes \wedge(M)$ 
\begin{equation}\label{Hodge}
\delta^{-1}\delta a + \delta\delta^{-1} a = a - a_{00}.
\end{equation}
Then we have now the standard theorem by Fedosov:
\begin{theorem}[Fedosov \cite{fed2}]\label{FedThm}
For any $\Omega \in \nu \Omega^2(M)[[\nu]]$ there exists a unique $r \in \Gamma(W)\otimes\wedge^1(M)$ with $\delta^{-1}r = 0$
and $W$-degree larger than 2 such that 
$$ \barR + \partial r - \delta r + \frac1\nu r\circ r = \Omega, $$
that is $D$ is Abelian.
\end{theorem}
\noindent
Note that the condition $\delta^{-1}r = 0$ implies the Weyl normalization condition $r|_{y=0} = 0$ and
the Weyl curvature is
\begin{equation*}\label{W-curvature2}
\Theta = - \omega + \Omega
\end{equation*}
where we put 
\begin{equation}\label{symp-form}
\omega =  \frac12 \omega_{ij} dx^i \wedge dx^j.
\end{equation}
The proof is given by showing that $r$ is obtained recursively in terms of its degrees by using
$$ r = \delta^{-1}(\barR - \Omega) + \delta^{-1}(\partial r + \frac1\nu r\circ r).$$

Thus $D$ is uniquely determined under the conditions $\delta^{-1}r = 0$ and $W$-degree of $r$ 
larger than 2 once we are given
a symplectic connection $\nabla$ and the central formal 2-form $\Omega$.
We call this connection $D$ the {\it Fedosov connection}.
Consider the space $\Gamma(W)_D$ of flat (or parallel) sections with respect to $D$, i.e. the sections $a$ with $Da = 0$. Then it is shown \cite{fed2} that
$\sigma : \Gamma(W)_D \to C^\infty(M)[[\nu]]$ sending $a \in \Gamma(W)_D$ to $a_{00} \in C^\infty(M)[[\nu]]$
is a bijection. Its inverse, denoted by $Q$, sending $a_{00}$ to $a:= Qa_{00}$ can be constructed by solving
recursively 
$$ a = a_{00} + \delta^{-1}(\partial a + \frac 1\nu [r,a])$$
since $\delta^{-1}$ increases the Weyl degree at least by $1$. $Q$ is explicitly expressed as
\begin{equation}\label{Q}
Qa_{00} = \sum_{k\ge 0} (\delta^{-1}(\partial + \frac1\nu[r,\cdot]))^k a_{00}.
\end{equation}
Since $D$ is a derivation of $\circ$, i.e. 
$$ D(a\circ b) = Da\circ b + a \circ Db,$$
$\Gamma(W)_D$  is closed under the product $\circ$. Then  
the product $\circ$ on $\Gamma(W)_D$ induces through $Q$ a $\ast$-product $\ast_{\nabla, \Omega}$ on $C^\infty(M)[[\nu]]$
which we call the {\it Fedosov star product}.

In the same way we can prove the following lemma.
\begin{lemma}\label{D_inverse}
Suppose $b \in \Gamma(W)\otimes \wedge^1(M)$ satisfy $Db = 0$. Then the equation $Da = b$ admits a
unique solution $a \in \Gamma(W)$, denoted by $a= D^{-1}b$, such that $a|_{y=0} = 0$.
\end{lemma}
\begin{proof}
We use the Hodge decomposition \eqref{Hodge} for our $a$. Since $a$ is a $0$-form we have $\delta^{-1}a = 0$, and also have 
$a_{00} = a|_{y=0} = 0$. Thus we have  
$$a = \delta^{-1}\delta a.$$ 
From this and the equation 
$$Da = \partial a - \delta a + \frac1\nu [r,a] = b$$ 
we need to solve 
\begin{equation*}
 a = - \delta^{-1} b + \delta^{-1} (\partial a + \frac1\nu[r,a]).
 \end{equation*}
 This can be solved recursively since $\delta^{-1}$ raises degree by $1$, and the solution is given explicitly 
using the same expression as $Q$ (Equation \eqref{Q})
in the form
\begin{equation}\label{D_inverse2}
a = \sum_{k\ge 0} (\delta^{-1}(\partial + \frac1\nu[r,\cdot]))^k(-\delta^{-1}b).
\end{equation}
\end{proof}

Later we will often use the relation 
\begin{equation}\label{D_inverse3}
D^{-1} = - Q\circ \delta^{-1}.
\end{equation}
Next we recall the following proposition due to Fedosov, see Proposition 5.5.5 and 5.5.6 in \cite{fed}. Since the characterization of $V$ is used in later arguments we re-produce its proof in this paper. 
\begin{proposition}[\cite{fed}]\label{equivalence}
On contractible Darboux chart $U$ we have an equivalence 
$$ A : \Gamma(W)_D|_U \to \Gamma(W|_U)_{D_{\mathrm{flat}}}$$
between the Fedosov connection $D$ and $D_{\mathrm{flat}} = d - \delta$
and 
this equivalence $A$ is expressed as 
$$ Aa = V\circ a \circ V^{-1}$$
for some $V \in \Gamma(W|_U)$.
\end{proposition}
\begin{proof}
We look for $A_s$ and $D_s$ interpolating between $D$ and $D_{\mathrm{flat}}$. By \eqref{r} we have 
$$ D = D_{\mathrm{flat}} + \frac1\nu[r,\cdot] $$
on $U$. We write the symplectic connection we chose as $d + \Gamma$ and the flat connection 
$\nabla_{\mathrm{flat}} = d$, and join them by $\nabla_s = d + (1-s)\Gamma$. Since $\nabla$ and 
$\nabla_{\mathrm{flat}}$ are both symplectic connections for the symplectic form in the Darboux chart $U$
the affine line $\nabla_s$ are symplectic connections for all $s$. We also set $\Omega_s := (1-s)\Omega$,
and build the Fedosov connection for $\nabla_s$ and $\Omega_s$ so that
the Weyl curvature of $D_s$ is $-\omega + \Omega_s = -\omega + (1-s)\Omega$. 
We define 
$r_s \in \Gamma(W)\otimes \wedge^1(M)$ by 
$$ D_s = D - \frac1\nu [r_s,\cdot].$$
Note that $r_0 = 0$. Since $ D = D_s + \frac1\nu [r_s,\cdot]$ and 
$$ D^2 = D_s^2 + \frac1\nu[D_s r_s,\cdot] + \frac1{\nu^2}[r_s\circ r_s,\cdot]$$
we have
\begin{equation}\label{curvature_s}
\Omega = (1-s)\Omega + D_s r_s + \frac1\nu r_s\circ r_s.
\end{equation}
Taking the derivative of \eqref{curvature_s} with respect to $s$ we obtain
\begin{equation}\label{curvature_s2}
-\Omega + D_s \dot{r}_s = 0
\end{equation}
since
$$ \dot{D}_s r_s = - \frac1\nu[\dot{r}_s,r_s] = - \frac1\nu(\dot{r}_s\circ r_s + r_s\circ \dot{r}_s).$$
On the contractible Darboux chart we can write the closed $2$-form $\Omega$ as
$$ \Omega = d\alpha$$
for some $1$-form $\alpha$. Then by using \eqref{curvature_s2} we have
\begin{eqnarray*}
D_s(-\alpha + \dot{r}_s) &=& -d\alpha + D_s \dot{r}_s\\
&=& -\Omega + D_s \dot{r}_s\\
&=& 0.
\end{eqnarray*}
By Lemma \ref{D_inverse} there is a unique solution $H(s) \in \Gamma(W|_U)$ such that 
$H(s)|_{y=0} = 0$ and
\begin{equation}\label{H(s)}
D_sH(s) = -\alpha + \dot{r}_s.
\end{equation}
Note that, from the construction of the Fedosov connection in Theorem \ref{FedThm} and 
\eqref{D_inverse2}, the degree $H(s)$ is at least $3$ as $\delta^{-1}$ raises the Weyl degree by $1$.
We then solve 
\begin{equation*}\label{V_s}
\frac{d}{ds} V_s = \frac1\nu H(s)\circ V_s
\end{equation*}
of $V_s \in \Gamma(W|_U)$ with $V_0 = 1$. This can be solved solving the integral equation 
$$ V_s = 1 + \int_0^s \frac1\nu H(\sigma)\circ V_\sigma d\sigma$$
recursively
using the Weyl degree. 
The iterations are completed since the integral operator on the right hand side raises the Weyl degree by $1$.
Then we have 
\begin{eqnarray*}
&&\frac{d}{ds} (V_s^{-1}\circ D_s V_s - \frac1\nu r_s) \\
&=& 
-\frac1\nu V_s^{-1}\circ H(s)\circ D_s V_s - V_s^{-1}\circ (\frac1\nu [\dot{r}_s,V_s]) + V_s^{-1}\circ D_s (\frac1\nu H(s)\circ V_s) - \frac1\nu \dot{r}_s\\
&=& \frac1\nu V_s^{-1}\circ (D_s H(s) - \dot{r}_s)\circ V_s\\
&=& -\frac1\nu V_s^{-1}\circ \alpha\circ V_s\\
&=& -\frac1\nu\alpha
\end{eqnarray*}
where we have used \eqref{H(s)} and the fact that $\alpha$ is central. This shows
\begin{equation}\label{V_s2}
V_s^{-1}\circ D_s V_s - \frac1\nu r_s = -\frac{s}\nu\alpha.
\end{equation}
Finally we define $A_s : \Gamma(W)_D|_U \to (\Gamma(W)|_U)_{D_s}$ by
$a \mapsto V_s\circ a \circ V_s^{-1}.$ Then we see for $a \in \Gamma(W)_D|_U$ that 
\begin{eqnarray*}\label{A_s1}
D_sA_s a &=& V_s\circ[V_s^{-1}\circ D_s V_s - \frac1\nu r_s, a]\circ V_s^{-1}\\
&=& V_s\circ[-\frac1\nu\alpha,a]\circ V_s^{-1}\nonumber\\
&=& 0 \nonumber
\end{eqnarray*}
since $\alpha$ is central. 
$V:=V_1$ is the one we desired. This completes the proof of Proposition \ref{equivalence}.
\end{proof}

Note in particular, \eqref{V_s2} shows 
\begin{equation}\label{V}
V^{-1}\circ D_{\mathrm{flat}} V - \frac1\nu r = -\frac{1}\nu\alpha.
\end{equation}

Now, since $B = \mathrm{ev}_{y=0}\circ A \circ Q$, the normalized trace \eqref{normalized_tr} can be expressed as
\begin{eqnarray}\label{tr}
\Tr^{\ast_{\nabla,\Omega}}(F) &=& (2\pi\nu)^{-m}\int_M (AQ(F))|_{y=0} \frac{\omega^m}{m!}\nonumber\\
&=& (2\pi\nu)^{-m}\int_M (V\circ Q(F)\circ V^{-1})|_{y=0} \frac{\omega^m}{m!}.
\end{eqnarray}

\section{Variation Formula}
Let $\nabla^t$ be a family of symplectic connections, and $\Omega_t$ a family of closed formal $2$-forms
in the same cohomology class. We write 
$$\Omega_t - \Omega_0 = d\beta_t.$$ 
For the pair of $\nabla^t$ and
$\Omega_t$ we have the Fedosov connection $D_t$ and the Fedosov star product $\ast_{\nabla^t,\Omega_t}$. 
We denote respectively by $\Tr^{\ast_{\nabla^t,\Omega_t}}$ and $\rho^{\nabla^t,\Omega_t}$ the trace and 
its trace density with
respect to $\ast_{\nabla^t,\Omega_t}$. We also use the notations $D_t^{-1}$ and $Q_t$ for $D^{-1}$ in 
Lemma \ref{D_inverse} and $Q$ in \eqref{Q} with respect to $t$. Let us write the variation formula for the trace, our proof follows Fedosov's paper \cite{fedtrace} incorporating variations of $\Omega$.
\begin{theorem}\label{variation1}
With the notations being as above we have for any formal function $F \in C^\infty(M)[[\nu]]$ 
$$ \frac{d}{dt} \Tr^{\ast_{\nabla^t,\Omega_t}}(F) = (2\pi\nu)^{-m}\int_M \frac1\nu [D_t^{-1}(\dot{\bar\Gamma} - \dot{\beta}), Q_t(F)]|_{y=0}\ 
\rho^{\nabla^t,\Omega_t} \frac{\omega^m}{m!}.$$
\end{theorem}
\begin{proof} Write locally on the contractible Darboux Chart $U$ as
\begin{equation}\label{D_t}
 D_t = D_\mathrm{flat} + \frac1\nu [r_t, \cdot].
 \end{equation}
(Note that this $r_t$ is different from $r_s$ in the previous section.) For each $t$, we built in the previous section
$$ A_t : (\Gamma(W)_{D_t})|_U \to \Gamma(W|_U)_{D_\mathrm{flat}}$$
such that 
$$ A_t(a) = V_t\circ a \circ V_t^{-1}$$
and
\begin{equation}\label{V_t}
V_t^{-1}\circ D_\mathrm{flat} V_t - \frac1\nu r_t = - \frac1\nu\alpha_t
\end{equation}
for some $V_t \in \Gamma(W)$ (see \eqref{V})  where $\Omega_t = d\alpha_t$ on $U$.
By \eqref{tr} the normalized trace can be expressed as
\begin{equation}\label{tr2}
\Tr^{\ast_{\nabla^t,\Omega_t}}(F) = (2\pi\nu)^{-m}\int_M \left(V_t\circ Q_t(F)\circ V_t^{-1}\right)|_{y=0} \frac{\omega^m}{m!}.
\end{equation}
Hereafter we omit the notation $\circ$. To compute the derivative of \eqref{tr2} with respect to $t$  we see
\begin{equation}\label{tr3}
\frac{d}{dt} (V_t Q_t(F) V_t^{-1}) = V_t([V_t^{-1}\dot{V}_t,Q_t(F)] + \dot{Q}_t(F))V_t^{-1}.
\end{equation}
We first treat $[V_t^{-1}\dot{V}_t,Q_t(F)] $. Taking the derivative of \eqref{V_t} with respect to $t$ we obtain
\begin{equation}\label{tr4}
[V_t^{-1}D_\mathrm{flat}V_t, V_t^{-1}\dot{V}_t] + D_\mathrm{flat}(V_t^{-1}\dot{V}_t) = \frac1\nu(\dot{r}_t - \dot{\alpha}_t).
\end{equation}
Using \eqref{V_t}, the fact that $\alpha_t$ is central and \eqref{D_t} we obtain from \eqref{tr4}
\begin{equation}\label{tr5}
D_t(V_t^{-1}\dot{V}_t) = \frac1\nu(\dot{r}_t - \dot{\alpha}_t).
\end{equation}
On the other hand, from
$$ D_t^2 = (D_\mathrm{flat} + \frac1\nu[r_t,\cdot])^2$$
we have
$$  \Omega_t = D_\mathrm{flat} r_t + \frac1\nu r_t\circ r_t$$
and thus
\begin{eqnarray}\label{solvable}
D_t(\dot{r}_t - \dot\alpha_t) &=& D_t\dot{r}_t - \dot{\Omega}_t \nonumber\\
&=& D_\mathrm{flat}\dot{r}_t + \frac1\nu[r_t,\dot{r}_t] - \dot\Omega_t = 0
\end{eqnarray}
Thus by Lemma \ref{D_inverse}, \eqref{tr5} and \eqref{solvable} it follows that 
\begin{equation*}\label{solvable2}
V_t^{-1}\dot{V}_t = D_t^{-1} (\dot{r}_t - \dot\alpha_t) + b
\end{equation*}
for some $b \in \Gamma(W)_{D_t}$. Note that this $b$ is necessary because for uniqueness we have 
to impose $V_t^{-1}\dot{V}_t|_{y=0} = 0$. 
Hence \eqref{tr3} has become
\begin{equation}\label{tr6}
\frac{d}{dt} (V_t Q_t(F) V_t^{-1}) = V_t\left(\frac1\nu[ D_t^{-1} (\dot{r}_t - \dot\alpha_t) + b,Q_t(F)] + \dot{Q}_t(F)\right)V_t^{-1}.
\end{equation}
Now we treat $\dot{Q}_t(F)$. Taking the derivative of $D_tQ_t(F) = 0$ we obtain
\begin{eqnarray*}
D_t\dot{Q}_t(F) = -\frac1\nu[\dot{r}_t - \dot{\alpha}_t,Q_t(F)]
\end{eqnarray*}
since $\dot{\alpha}_t$ is central. Using \eqref{solvable} again we have
$$ D_t[\dot{r}_t - \dot{\alpha}_t,Q_t(F)] = 0.$$
Since $Q_t(F)|_{y=0} = F$ we also have the uniqueness condition $\dot{Q}_t(F)|_{y=0} = 0$.
Thus by Lemma \ref{D_inverse}
$$ \dot{Q}_t(F) = - \frac1\nu D_t^{-1} [\dot{r}_t - \dot{\alpha}_t,Q_t(F)].$$
Now \eqref{tr6} has become
\begin{equation*}\label{tr7}
\frac{d}{dt} (V_t Q_t(F) V_t^{-1}) = V_t\left(\frac1\nu[ D_t^{-1} (\dot{r}_t - \dot\alpha_t) + b,Q_t(F)] -  \frac1\nu D_t^{-1} [\dot{r}_t - \dot{\alpha}_t,Q_t(F)]\right)V_t^{-1}.
\end{equation*}
Thus we obtain
\begin{eqnarray*}
\frac{d}{dt} \Tr^{\ast_{\nabla^t,\Omega_t}}(F) &=& \int_M V_t\left(\frac1\nu[ D_t^{-1} (\dot{r}_t - \dot\alpha_t) + b,Q_t(F)] -  \frac1\nu D_t^{-1} [\dot{r}_t - \dot{\alpha}_t,Q_t(F)]\right)V_t^{-1}|_{y=0} \frac{\omega^m}{m!}\\
&=& \Tr^{\ast_{\nabla^t,\Omega_t}}\left(\frac1\nu[ D_t^{-1} (\dot{r}_t - \dot\alpha_t) + b,Q_t(F)] -  \frac1\nu D_t^{-1} [\dot{r}_t - \dot{\alpha}_t,Q_t(F)]\right)|_{y=0}).
\end{eqnarray*}
Recall by \eqref{D_inverse2} that 
$$ D_t^{-1}|_{y=0} = -Q(\delta^{-1}\cdot)|_{y=0} = 0$$
since $\delta^{-1}$ increases $y$-degree by $1$. Thus 
\begin{equation}\label{final1}
\frac{d}{dt} \Tr^{\ast_{\nabla^t,\Omega_t}}(F) 
= \Tr^{\ast_{\nabla^t,\Omega_t}}\left(\frac1\nu[ D_t^{-1} (\dot{r}_t - \dot\alpha_t) + b,Q_t(F)]|_{y=0}\right).
\end{equation}
Also $b = Q_t(b_0)$ for some $b_0 \in C^\infty(M)$, and by the property of the trace we have
$$ \Tr^{\ast_{\nabla^t,\Omega_t}}([b,Q_t(F)]|_{y=0}) = \Tr^{\ast_{\nabla^t,\Omega_t}}([b_0,F]_{\ast_{\nabla^t,\Omega_t}}) = 0.$$
Thus \eqref{final1} becomes
\begin{equation*}\label{final2}
\frac{d}{dt} \Tr^{\ast_{\nabla^t,\Omega_t}}(F) 
= \Tr^{\ast_{\nabla^t,\Omega_t}}\left(\frac1\nu[ D_t^{-1} (\dot{r}_t - \dot\alpha_t),Q_t(F)]|_{y=0}\right).
\end{equation*}
Recall also \eqref{D_inverse3} so that 
$$ D_t^{-1}\dot{r}_t = -Q(\delta^{-1}\dot{r}_t).$$
Our $\dot{r}_t$ comes from the variation $\dot{\bar{\Gamma}}$ of the symplectic connection and the variation of the $r$-term in
Fedosov's construction in Theorem \ref{FedThm}. But the $r$-term in Fedosov's construction is required
$\delta^{-1}r = 0$. Hence we have
$$ D_t^{-1}\dot{r}_t = D_t^{-1}\dot{\bar\Gamma}.$$
\noindent
The Theorem \ref{variation1} follows by noting $\dot{\alpha} = \dot{\beta}$. This completes the proof.
\end{proof}

\section{Quantum moment map}

The formula in the next proposition can be found in page 135 in \cite{gr3}, but we will re-produce its proof
as we wish to make clear how the assumptions are used.
\begin{proposition}[\cite{gr3}]\label{Lie der}
For any triple $(\omega, \Omega, \nabla) \in \mathcal C^G([\omega_0],[\Omega_0])$ and $X \in \mathfrak g$
we have the identity
$$ L_X = D\circ i(X) + i(X)\circ D + \frac1\nu \ad_\ast(Q(\mu_X)).$$
\end{proposition}
\begin{proof}
We start with the general formula in page 135, \cite{gr3}, 
for the Fedosov star product and a symplectic vector field $X$ i.e. $di(X)\omega = 0$
\begin{equation}\label{QM0}
 L_X = D\circ i(X) + i(X)\circ D + \frac1\nu \ad_\ast(T(X))
 \end{equation}
where
$$ T(X) = - i(X)r + \omega_{ij}X^iy^j + \frac12(\nabla_i(i(X)\omega)_jy^iy^j),$$
but note that the sign of $r$ in \cite{gr3} is opposite from ours.
To prove the proposition it is sufficient to show
\begin{equation}\label{QM1}
D(\mu_X + T(X)) = 0.
\end{equation}
First of all, since $L_X\nabla=0$ and $L_X\omega = L_X\Omega = 0$ we have $L_X r = 0$. Thus by \eqref{QM0} we have
\begin{equation}\label{QM2}
 -D i(X)r =  i(X)Dr + \frac1\nu [T(X), r].
 \end{equation}
 From Theorem \ref{FedThm} and \eqref{r} we see
 \begin{equation}\label{QM3}
Dr = -\overline{R} + \Omega +\frac1{2\nu}[r,r].
\end{equation}
From \eqref{QM2} and \eqref{QM3} we obtain
 \begin{equation}\label{QM4}
-D i(X)r = -i(X) \overline{R} + i(X)\Omega -\frac1{\nu}[\omega_{ij}X^iy^j + \frac12(\nabla_i(i(X)\omega)_jy^iy^j),r].
\end{equation}
Secondly, using \eqref{r} we obtain
 \begin{equation}\label{QM5}
D(\omega_{ij}X^iy^j) = -i(X)\omega + \partial (\omega_{ij}X^iy^j) + \frac1{\nu}[\omega_{ij}X^iy^j,r].
\end{equation}
Thirdly, using \eqref{r} again we have
 \begin{eqnarray}\label{QM6}
&&D(\frac12(\nabla_i(i(X)\omega)_jy^iy^j)) \\
&&= -\nabla_i(i(X)\omega)_j dx^i y^j + \partial (\frac12(\nabla_i(i(X)\omega)_jy^iy^j)) + \frac1{\nu}[\frac12(\nabla_i(i(X)\omega)_jy^iy^j),r].\nonumber
\end{eqnarray}
The condition that $\nabla$ is $G$-invariant implies $L_X\nabla = 0$, which is equivalent to say
\begin{equation*}
(\nabla^2X)(Y,Z) =( -\nabla_X\nabla_Y  + \nabla_Y\nabla_X + \nabla_{[X,Y]})Z,
\end{equation*}
or equivalently
\begin{equation*}
\nabla_k\nabla_iX^p = - \mathrm{R}^p{}_{i\ell k}X^\ell
\end{equation*}
where $\ell,\ k$ are regarded as indices of the form part and $p,\ i$ are regarded as the indices of the endomorphism part. 

From this and \eqref{symp-form} we obtain
\begin{eqnarray*}
\partial (\frac12(\nabla_i(i(X)\omega)_jy^iy^j))&=&\frac14 \nabla_k\nabla_iX^p\omega_{pj} y^iy^j dx^k\nonumber\\
&=& -\frac14  \mathrm{R}^p{}_{i\ell k}X^\ell \omega_{pj}y^iy^j dx^k\nonumber\\
&=& \frac14  \omega_{jp}\mathrm{R}^p{}_{i\ell k} y^iy^j X^\ell dx^k\nonumber\\
&=& i(X)\overline{R}.
\end{eqnarray*}
Thus \eqref{QM6} becomes
 \begin{eqnarray}\label{QM7}
&&D(\frac12(\nabla_i(i(X)\omega)_jy^iy^j)) \\
&&= -\nabla_i(i(X)\omega)_j dx^i y^j + i(X)\overline{R} + \frac1{\nu}[\frac12(\nabla_i(i(X)\omega)_jy^iy^j),r].\nonumber
\end{eqnarray}
Adding \eqref{QM4}, \eqref{QM5} and \eqref{QM7} we obtain
 \begin{eqnarray*}
DT &=& i(X)(-\omega + \Omega) \\
&=& -d\mu_X = -D\mu_X.
\end{eqnarray*}
This shows \eqref{QM1} completing the proof of Proposition \ref{Lie der}.
\end{proof}

\begin{proposition}\label{norma_var} The following two hold about normalization.\\
\begin{enumerate}
\item[(a)] For a quantum Hamiltonian vector field $X$, the quantum Hamiltonian function is determined 
uniquely under the normalization condition \eqref{normalize_c}.
\item[(b)] Let $(M, \omega_t,\Omega_t)$ be quantum Hamiltonian $G$-spaces, for $t\in I$ such that 
\begin{equation}\label{tau}
 \omega_t - \Omega_t = \omega_0 - \Omega_0 + d\tau_t
 \end{equation}
for a smooth family of $G$-invariant formal $1$-form $\tau_t$. Then the normalized quantum Hamiltonian functions $u_{X,t}$ for $\omega_t - \Omega_t$ with normalization condition \eqref{normalize_c} are related by
\begin{equation*}
u_{X,t} = u_{X,0} - \tau_t(X).
\end{equation*}
\end{enumerate}
\end{proposition}
\begin{proof} The statement of item (a) is obvious because if we have two quantum Hamiltonian functions of the same vector field $X$ 
then the difference of the two is a formal constant.

To show (b) one can see $i(X)(\omega_t - \Omega_t) = du_{X,t}$ and that $u_{X,t}$ is independent of the choice of
the $G$-invariant $1$-form $\tau_t$ satisfying \eqref{tau}
since another $\tau'_t$ satisfying \eqref{tau} is of the form $\tau'_t = \tau_t + dh_t$ for a $G$-invariant smooth function $h_t$. One further sees
\begin{eqnarray*}
\frac{d}{dt} \int_M u_{X,t}(\omega_t - \Omega_t)^m &=& - \int_M \dot{\tau}(X) (\omega_t - \Omega_t)^m
+ \int_M m u_{X,t} d\dot{\tau} \wedge (\omega_t - \Omega_t)^{m-1}\\
&=& -\int_M \dot{\tau} \wedge i(X) (\omega_t - \Omega_t)^m - \int_M m du_{X,t} \wedge \dot{\tau} \wedge (\omega_t - \Omega_t)^{m-1}\\
&=& 0.
\end{eqnarray*}
Thus if $u_{X,0}$ satisfies normalization \eqref{normalize_c} then so does $u_{X,t}$ for all $t$. This proves (b). 
\end{proof}

\begin{proof}[Proof of Theorem \ref{main thm}.]
As Step 1, we consider the case when we have $(\omega_0, \Omega_0, \nabla^0)$ and 
$(\omega_0, \Omega, \nabla) \in \mathcal C^G([\omega_0],[\Omega_0])$. 
We take a family $(\omega_0, \Omega_t, \nabla^t)$ in 
$\mathcal C^G([\omega_0],[\Omega_0])$ joining $(\omega_0, \Omega_0, \nabla^0)$ and $(\omega_0, \Omega, \nabla)$. 
We put
$$\Omega_t = \Omega_0 + d\beta_t$$
for a $G$-invariant formal $1$-form $\beta_t$. Then by Proposition \ref{norma_var}, the quantum Hamiltonian function $\mu_{X,t}$ for $\omega_0 - \Omega_t$ with normalization \eqref{normalize_c} is given by
$$ \mu_{X,t} = \mu_X + \beta_t(X).$$
By Theorem \ref{variation1} we have
\begin{equation}\label{trace_variation}
\frac{d}{dt} \Tr^{\ast_{\nabla^t,\Omega_t}}(\mu_{X,t}) = (2\pi\nu)^{-m}\int_M \left(\frac1\nu [D_t^{-1}(\dot{\bar\Gamma} - \dot{\beta}), Q_t(\mu_{X,t})]|_{y=0} + \dot{\beta}(X)\right)
\rho^{\nabla^t,\Omega_t} \frac{\omega_0^m}{m!}.
\end{equation}
By Proposition \ref{Lie der} we have
$$ \frac1\nu [D_t^{-1}(\dot{\bar\Gamma} - \dot{\beta}), Q_t(\mu_{X,t})]|_{y=0} 
= (-L_X + D_t\circ i(X) + i(X)D_t) D_t^{-1}(\dot{\bar\Gamma} - \dot{\beta})|_{y=0} .$$
But $D_t^{-1}(\dot{\bar\Gamma} - \dot{\beta})$ is a $0$-form, i.e. a function, so that
$$ i(X) D_t^{-1}(\dot{\bar\Gamma} - \dot{\beta}) = 0.$$
Further, recall $D_t^{-1} = - Q_t\circ\delta^{-1}$ by \eqref{D_inverse3} and $\delta^{-1}$ increases $y$-degree by $1$ so that
$$ L_X D_t^{-1}(\dot{\bar\Gamma} - \dot{\beta})|_{y=0} = 0.$$
The remaining term becomes 
\begin{eqnarray*}
i(X)D_t D_t^{-1}(\dot{\bar\Gamma} - \dot{\beta})|_{y=0} &=& i(X) (\dot{\bar\Gamma} - \dot{\beta})|_{y=0} \\
&=& - \dot{\beta}(X)
\end{eqnarray*}
since $\dot{\bar\Gamma}$ has $y$-degree $2$. Thus the right hand side of \eqref{trace_variation} vanishes.
Thus 
$$\Tr^{\ast_{\nabla,\Omega}}(\mu_{X}) = \Tr^{\ast_{\nabla^0,\Omega_0}}(\mu_{X,0})$$
for fixed $\omega = \omega_0$. 
This completes the proof of Theorem \ref{main thm} in the case when $\omega$ is fixed to be $\omega_0$.

As Step 2, we consider the case when we have $(\omega_0, \Omega_0, \nabla^0)$ and $(\omega, \Omega, \nabla) \in \mathcal C^G([\omega_0],[\Omega_0])$.
Then 
there is a smooth path $\{\omega_s\}_{0 \le s \le 1}$ consisting of $G$-invariant
symplectic forms joining $\omega_0$ and $\omega_1 = \omega$ in the cohomology class $[\omega_0]$
Then we have $G$-equivariant 
diffeomorphisms $f_s : M \to M$ such that $f^\ast_s\omega_s = \omega_0$ by Moser's theorem. 
We put $f:= f_1$ for notational convenience. Then we have 
$$ f^\ast(\omega, \Omega, \nabla) = (\omega_0, f^\ast\Omega, f^\ast\nabla)$$
with $f^\ast\Omega$ cohomologous to $\Omega_0$
and $f^\ast\nabla$ being a symplectic connection for $f^\ast\omega = \omega_0$. 
Then we are in a position where the same arguments as in Step 1 apply for the pair $(\omega_0, \Omega_0, \nabla^0)$
and $(\omega_0, f^\ast\Omega, f^\ast\nabla)$. 
We obtain from Step 1
$$
\Tr^{\ast_{\omega_0,\nabla^0,\Omega_0}}(\mu_{X,0}) = \Tr^{\ast_{\omega_0,f^\ast\nabla,f^\ast\Omega}}(f^\ast\mu_{X})
$$
where $\mu_{X,0}$ indicates the quantum moment map for $\omega_0 - \Omega_0$ and
where we indicated the symplectic forms with respect to which the star products are considered.
But since $f$ is a $G$-equivariant symplectomorphism the right hand side is equal to 
$$ 
\Tr^{\ast_{\omega_0,f^\ast\nabla,f^\ast\Omega}}(f^\ast\mu_{X})= \Tr^{\ast_{\omega,\nabla,\Omega}}(\mu_X).
$$
Thus $\Tr^{\ast_{\omega,\nabla,\Omega}}(\mu_X)$ is independent of $(\omega, \nabla, \Omega) \in \mathcal C^G([\omega_0],[\Omega_0])$
with the normalization condition \eqref{normalize_c} of $\mu_X$. 
This completes
the proof of Theorem \ref{main thm}.
\end{proof}

We now apply Theorem \ref{main thm} to the K\"ahler situation. Consider a K\"ahler manifold $(M,\omega_0,J)$ with 
$G$-invariant $\omega_0$ and $J$. For $\omega \in \calM^G_{[\omega_0]}$  we take the 
Levi-Civita connection $\nabla$. Then $(\omega,0,\nabla)$ is in $\mathcal C^G([\omega_0],0)$ and
the quantum Hamiltonian $G$-space $(M, \omega, 0)$ has quantum moment map $\mu_{\cdot}$ normalized by \eqref{normalize_c}, 
which means that for $X \in \mathfrak{g}$ the normalization gives 
$$\int_M \mu_X\, \omega^m=0$$
since we are taking $\Omega = 0$.

\begin{proposition} \label{prop:quantumkahler} Under the above situation the following two hold.
\begin{enumerate}
\item $\mu_X-\frac{\nu\,k}{2}\Delta^{(\omega)} \mu_X$ is a quantum-Hamiltonian with respect to the star product $*_{\nabla,\Omega_{k}(\omega)}$,
with $\Delta^{(\omega)}$ being the Laplacian with respect to $(\omega,J)$. In particular, $(\omega,\Omega_k(\omega),\nabla)\in \mathcal C^G([\omega_0],[\Omega_k(\omega_0)])$.
\item The integral 
\begin{equation}\label{eq:normalization}
\int_M \left(\mu_X-\frac{\nu\,k}{2}\Delta^{(\omega)} \mu_X\right)\, \left(\omega-\Omega_k(\omega)\right)^m
\end{equation}
is independent of the choice of $\omega \in \calM^G_{[\omega_0]}$.
\end{enumerate}
\end{proposition}

Before going to the proof, let us recall a particular case of the construction from \cite{Fut04}. On a compact
K\"ahler manifold $(M,\omega,J)$, consider the holomorphic bundle $T^{(1,0)}M$ consisting of 
tangent vectors of type $(1,0)$. Choose any $(1,0)$-connection $\overline{\nabla}$ on $T^{(1,0)}M$ 
with curvature $R^{\overline{\nabla}}$. For $Z$ in $\mathfrak{h}$ the reduced Lie algebra of holomorphic vector fields, 
define $L(Z^{(1,0)}):=\overline{\nabla}_{Z^{(1,0)}}-\mathcal{L}_{Z^{(1,0)}}$, it is a $0$-form with values 
in $\mathrm{End}(T^{(1,0)}M)$. Let $q$ be a $\mathrm{Gl}(m,\bfC)$-invariant polynomial on $\mathfrak{gl}(m,\bfC)$ of 
degree $p$, the first author defined in \cite{Fut04}, the map $\mathfrak{F}_q:\mathfrak{h}\rightarrow \bfC$ by
\begin{equation} \label{eq:Futakigendef}
\mathfrak{F}_q(Z):=  \int_M   -(m-p+1)u_Z q(R^{\overline{\nabla}})\wedge \omega^{(m-p)}+ q(L(Z^{(1,0)})+R^{\overline{\nabla}})\wedge \omega^{(m-p+1)} ,
\end{equation}
where $u_Z=f+ih\in C^{\infty}_0(M,\bfC)$ for $Z=X_f+JX_h\in \mathfrak{h}$. Remark that as $L(Z^{(1,0)})+R^{\overline{\nabla}}$ is a form of mixed degree, the form $q(L(Z^{(1,0)})+R^{\overline{\nabla}})$ in the second term of $\mathfrak{F}_q$ is also of mixed degree but only the component of degree $2(p-1)$ will contribute to the integral. 

One shows $\mathfrak{F}_q$ depends neither on the choice of the ${(1,0)}$-connection nor on the choice of the K\"ahler form in $\calM_{[\omega]}$, see \cite{Fut04}. 

\begin{lemma} \label{lemma:Futakic1}
For the polynomial $q:=(c_1)^p$, $\overline{\nabla}=\nabla$ the Levi-Civita connection and $Z=X_{f}\in \mathfrak{h}$ with $f\in C^{\infty}_0(M)$, the invariant $\mathfrak{F}_q$ writes as:
\begin{equation*}
\mathfrak{F}_{(c_1)^p }(Z)= \left(\frac{1}{2\pi}\right)^p\int_M - (m-p+1)f\Ric(\omega)^p \wedge \omega^{(m-p)} - \frac{p}{2}\,\Delta^{(\omega)} f \Ric(\omega)^{(p-1)}\wedge \omega^{(m-p+1)}
\end{equation*}
\end{lemma}

\begin{proof}
As $c_1(\cdot):=\frac{i}{2\pi}\tr^{\bfC}(\cdot)$ then $c_1(R^{\nabla})=\frac{1}{2\pi} \Ric(\omega)$ and as $u_Z=f$, the first term of the statement comes from the first term of the general formula \eqref{eq:Futakigendef}.
For the second term, we have
$$\int_M   (c_1)^p (L(Z^{(1,0)})+R^{\nabla})\wedge \omega^{(m-p+1)}= \int_M p\, c_1(L(Z^{(1,0)}))(c_1)^{p-1}(R^{\nabla})\wedge  \omega^{(m-p+1)}.$$
Now, $c_1(L(Z^{(1,0)}))=\frac{1}{2\pi}\tr^{\bfC}\left(Y^{(1,0)}\mapsto \nabla_{Y^{(1,0)}}X_f^{(1,0)}\right)=-\frac{1}{4\pi}\Delta^{\omega} f$.
\end{proof}

 \noindent Now, we can prove proposition \ref{prop:quantumkahler}.

\begin{proof}[Proof of Proposition \ref{prop:quantumkahler}]
\begin{enumerate}
\item It comes from $i_X \Ric(\omega)=d(\frac{1}{2}\Delta f)$ for $i_X\omega=df$ and $\mathcal{L}_X J=0$.
\item We compute the terms of order $\nu^p$ in the integral \eqref{eq:normalization}:
\begin{itemize}
\item at $p=0$, we have $\int_M \mu_X\, \omega^m=0$. Notice that this is the normalization \eqref{normalize_c} since $\Omega = 0$ for our quantum-Hamiltonian $G$-space $(M, \omega, 0)$.
\item at $p=1$, we have $-nk\int_M \mu_X \, \Ric(\omega)\wedge \omega^{(m-1)}$, which is the original Futaki invariant (the Laplacian does not contribute to the integral).
\item at $p>1$, we have
$$(-1)^p \frac{k^p}{m-p+1} {m\choose p}\int_M (m-p+1)\mu_X\Ric(\omega)^p \wedge \omega^{(m-p)} + \frac{p}{2}\,\Delta^{(\omega)} \mu_X \Ric(\omega)^{(p-1)}\wedge \omega^{(m-p+1)},$$
which is a K\"ahler invariant by Lemma \ref{lemma:Futakic1}
\end{itemize}
\end{enumerate}
\end{proof}

We set $\mu^k_{X}$ to be the quantum moment map with respect to $*_{\nabla,\Omega_k(\omega)}$ normalized by \eqref{normalize_c}, i.e.
\begin{equation*}
 \int_M \mu^k_{X} (\omega - \Omega_k(\omega))^m = 0,
\end{equation*}
and $\tilde{\mu}^k_{X}:= \mu_X-\frac{\nu\,k}{2}\Delta^{(\omega)} \mu_X$ which is normalized by \eqref{eq:normalizedkahler}, i.e. 
\begin{equation*}
 \int_M \tilde{\mu}^k_{X} \ \omega^m = 0.
\end{equation*}

\begin{proof}[Proof of Theorem \ref{main thm2}]
Let us compute for $X\in \mathfrak{g}$:
\begin{equation}
\Tr^{*_{\nabla,\Omega_k(\omega)}}(\tilde{\mu}^k_X)  = \Tr^{*_{\nabla,\Omega_k(\omega)}}(\mu^k_X)+ \Tr^{*_{\nabla,\Omega_k(\omega)}}(\tilde{\mu}^k_X-\mu^k_X)  .
\end{equation}
The first term of the right hand side is an invariant by Theorem \ref{main thm}. About the second term, note that 
$$\tilde{\mu}^k_X-\mu^k_X=\frac{1}{\int_M \left(\omega-\Omega_k(\omega)\right)^m}\int_M \left(\mu_X-\frac{\nu\,k}{2}\Delta^{(\omega)} \mu_X\right)\, \left(\omega-\Omega_k(\omega)\right)^m$$
which is an invariant by Proposition \ref{prop:quantumkahler}. Hence,
$$\Tr^{*_{\nabla,\Omega_k(\omega)}}(\tilde{\mu}^k_X-\mu^k_X)=
(\tilde{\mu}^k_X-\mu^k_X)\,\Tr^{*_{\nabla,\Omega_k(\omega)}}(1)$$
is an invariant since $\Tr^{*_{\nabla,\Omega_k(\omega)}}(1)$ is a topological invariant by the index theorem \cite{NT}.
This proves the first statement of Theorem \ref{main thm2}.
Moreover, if $*_{\nabla,\Omega_k(\omega)}$ is closed, then the trace density is $1$ and 
$$\Tr^{*_{\nabla,\Omega_k(\omega)}}(\tilde{\mu}^k_X)=(2\pi\nu)^{-m}\int_M\tilde{\mu}^k_X\,\omega^m=0.$$
This completes the proof of Theorem \ref{main thm2}.
\end{proof}

\section{Computations by hand up to $\nu^2$}

We are able to compute by hand the first order terms of the invariants $\Tr^{[\omega],[\Omega]}$ and $\Tr^{\calM^G_{[\omega]},k}$. To this end we compute the trace density up to order $\nu^2$.

\begin{proposition} \label{prop:densitynu2}
For the symplectic connection $\nabla$ and the formal $2$-form $\Omega:=\nu\alpha_1+\nu^2\alpha_2+O(\nu^3)$,
denote by $\rho^{\nabla,\Omega}:=\frac{1}{(2\pi\nu)^m}\left(1+\nu \rho_1+\nu^2 \rho_2+O(\nu^3)\right)$ the trace density of the Fedosov star product $*_{\nabla,\Omega}$. We have:
\begin{eqnarray*}
\rho_1 & := & -m\frac{\alpha_1\wedge\omega^{m-1}}{\omega^m} \\
\rho_2 & := & -\frac{1}{24} \mu(\nabla)-m\frac{\alpha_2\wedge\omega^{m-1}}{\omega^m} + \frac{1}{2}m(m-1)\dfrac{\alpha_1\wedge\alpha_1\wedge\omega^{m-2}}{\omega^m}
\end{eqnarray*}
for $\mu(\nabla)$ being the Cahen-Gutt momentum of $\nabla$.
\end{proposition}

\begin{proof}
Performing the Fedosov construction with symplectic connection $\nabla$ and the formal $2$-form $\Omega:=\nu\alpha_1+\nu^2\alpha_2+O(\nu^3)$, one obtains \cite{bordemann}:
\begin{equation*}
f*_{\nabla,\Omega}g  =  f.g+\frac{\nu}{2}\{f,g\}+\nu^2C_2(f,g) +\nu^3 C_3(f,g)+O(\nu^4)
\end{equation*}
with :
\begin{eqnarray*} 
C_2(f,g) & = &  \frac{1}{8}\Lambda^{i_1j_1} \Lambda^{i_2j_2} \nabla^2_{i_1i_2} f\nabla^2_{j_1j_2} g- \frac{1}{2} \alpha_1(X_f,X_g) \\
C_3(f,g) & = & \frac{1}{48}S^3_{\nabla}(f,g) +\frac{1}{2}(\imath_{X_f}\alpha_1)_i \Lambda^{ik} (\imath_{X_g}\alpha_1)_k  - \frac{1}{2}\alpha_2(X_f,X_g) + B^3_{\nabla}[\alpha_1](f,g)
\end{eqnarray*}
where 
\begin{equation}
S^3_{\nabla}(f,g)  =   \Lambda^{i_1j_1} \Lambda^{i_2j_2} \Lambda^{i_3j_3}L_{X_f}\nabla_{i_1i_2i_3}L_{X_g}\nabla_{j_1j_2j_3},
\end{equation}
for $L_{X_f}\nabla_{i_1i_2i_3}$ being the component of the Lie derivative of $\nabla$ seen as a symmetric $3$-tensor on $M$, and
\begin{eqnarray*}
B^3_{\nabla}[\alpha_1](f,g)&:=& \frac{1}{32} \left(\Lambda^{ta}(\alpha_1)_{au}\Lambda^{ui}\Lambda^{kj}+\Lambda^{ta}(\alpha_1)_{au}\Lambda^{uj}\Lambda^{ki}\right) \left(\nabla^2_{tk}f \nabla^2_{ij}g+\nabla^2_{tk}g \nabla^2_{ij}f\right) \nonumber\\
& & + \frac{1}{48} \left(\Lambda^{ui}\Lambda^{kj}+\Lambda^{uj} \Lambda^{ki}\right) \left((\imath_{X_f}\nabla_k \alpha_1)_u \nabla^2_{ij}g+\nabla^2_{ij}f(\imath_{X_g}\nabla_k \alpha_1)_u \right).
\end{eqnarray*}
Note that in \cite{bordemann} the conventions are slightly different from here: the formal parameter is rescaled by a factor $2$ as well as the formal $2$-form $\Omega$.

As $B^3_{\nabla}[\alpha_1](f,g)$ and $\Lambda^{i_1j_1} \Lambda^{i_2j_2} \nabla^2_{i_1i_2} f\nabla^2_{j_1j_2} g$ are symmetric in $f,g$ and the other terms of $C_3(f,g)$ are anti-symmetric in $f,g$, we have:
$$[f,g]_{\ast_{\nabla,\Omega}}= \nu \{f,g\}-\nu^2\alpha_1(X_f,X_g)+\nu^3C^-_3(f,g)+O(\nu^4),$$
where
$$C^-_3(f,g)=\frac{1}{24} S^3_{\nabla}(f,g)+ (\imath_{X_f}\alpha_1)_i \Lambda^{ik} (\imath_{X_g}\alpha_1)_k-\alpha_2(X_f,X_g). $$

The fact that $\rho_1$ and $\rho_2$ are the first terms of the trace density is summurised in the following equations: For $\rho_1$ we have
$$-\int\alpha_1(X_f,X_g)\frac{\omega^{m}}{m!} = \int_M \{f,g\}\frac{\alpha_1\wedge\omega^{m-1}}{(m-1)!}.$$
For $\rho_2$, first by the moment map property \cite{cagutt} of $\mu$ we have
$$\int S^3_{\nabla}(f,g)\frac{\omega^{m}}{m!}= \int \{f,g\}\mu(\nabla) \frac{\omega^{m}}{m!},$$
also, 
$$-\int\alpha_2(X_f,X_g)\frac{\omega^{m}}{m!}=\int_M \{f,g\}\frac{\alpha_2\wedge\omega^{m-1}}{(m-1)!},$$
and finally,
$$\frac{1}{2}\int_M \{f,g\}\frac{\alpha_1\wedge\alpha_1\wedge\omega^{m-2}}{(m-2)!}=\int  \alpha_1(X_f,X_g)\rho_1\frac{\omega^{m}}{m!}- \int(\imath_{X_f}\alpha_1)_i \Lambda^{ik} (\imath_{X_g}\alpha_1)_k\frac{\omega^{m}}{m!}.$$

\end{proof}

\begin{remark} \label{remark:rho1kahler}
On a K\"ahler manifold $(M,\omega,J)$, applying Proposition \ref{prop:densitynu2} to $\Omega=\Omega_k(\omega)$ yields 
$$\rho^{\nabla,\Omega_k(\omega)}=\frac{1}{(2\pi\nu)^m}\left(1-\frac{\nu\,k}{2} \,\mathrm{S}_{\omega} + O(\nu^2)\right)$$
\end{remark}

\noindent We can now compute the first terms of the invariants.

\begin{proposition}\label{prop:firstinvariantcomp}
The invariant
$$\Tr^{[\omega],[\Omega]}(X)=-\frac{1}{(2\pi)^m\nu^{m-2}24} \int \mu_X^0 \mu(\nabla) \frac{\omega^m}{m!}+ O(\nu^{m-3}),$$
for $\mu_X:=\mu_X^0+\nu\mu_X^1+\nu^2\mu_X^2+O(\nu^3)$, the quantum moment map normalised by \eqref{normalize_c}.
\end{proposition}

\begin{proof}
We compute the trace
\begin{eqnarray*}
(2\pi\nu)^m\Tr^{[\omega],[\Omega]}(X) & = & (2\pi\nu)^m\Tr^{*_{\nabla,\Omega}}(\mu_X) \\
 & = &\int \mu_X^0 \frac{\omega^{m}}{m!}+\nu\left( -\int \mu_X^0 \frac{\alpha_1\wedge\omega^{m-1}}{(m-1)!} + \int \mu_X^1 \frac{\omega^{m}}{m!}\right)\\
 & & +\nu^2\left(-\frac{1}{24}\int \mu^0_X \mu(\nabla)\frac{\omega^{m}}{m!}- \int \mu_X^0 \frac{\alpha_2\wedge\omega^{m-1}}{(m-1)!}+ \int \mu_X^0 \frac{1}{2}\dfrac{\alpha_1\wedge\alpha_1\wedge\omega^{m-2}}{(m-2)!}\right. \\
 & & \left.-\int\mu_X^1 \frac{\alpha_1\wedge\omega^{m-1}}{(m-1)!} +\int \mu_X^2 \frac{\omega^{m}}{m!}\right) + O(\nu^3)
\end{eqnarray*}
But the moment map is normalised so that $\int \mu_X (\omega-\Omega)^n=0$, hence the above becomes:
$$(2\pi\nu)^m\Tr^{[\omega],[\Omega]}(X)=-\frac{\nu^2}{24}\int \mu^0_X \mu(\nabla)\frac{\omega^{m}}{m!}+ O(\nu^3).$$
\end{proof}

\begin{proposition}
The invariant
$$\Tr^{\calM^G_{[\omega]},k}(X)=\frac{1}{(2\pi\nu)^m}\left(\frac{2\pi\nu}{m!}\mathfrak{F}_{kc_1}(X)+\frac{(2\pi\nu)^2}{(m-1)!}\mathfrak{F}_{\frac{1}{12}c_2-\frac{1+12k^2}{24}c_1^2}(X)+O(\nu^3)\right).$$
\end{proposition}

\begin{proof}
We consider the moment map $\tilde{\mu}^k_X:=\mu_X-\frac{\nu k}{2}\Delta^{(\omega)} \mu_X$ given in Proposition \ref{prop:quantumkahler}, normalised by the integral. Using Proposition \ref{prop:densitynu2} with $\Omega=\Omega_k(\omega)$, the trace is 
\begin{eqnarray*}
(2\pi\nu)^m\Tr^{\calM^G_{[\omega]},k}(X) & = & - k \nu  \int \mu_X \frac{\Ric(\omega)\wedge\omega^{m-1}}{(m-1)!} + \nu^2 \left(-\frac{1}{24}\int \mu_X\mu(\nabla) \frac{\omega^{m}}{m!} \right.\\
 & &\left.+ \frac{k^2}{2}\int \mu_X \dfrac{\Ric(\omega)\wedge\Ric(\omega)\wedge\omega^{m-2}}{(m-2)!}+\Delta^{(\omega)} \mu_X \frac{\Ric(\omega)\wedge\omega^{m-1}}{(m-1)!} \right)+O(\nu^3)
\end{eqnarray*}
The term in $\nu$ is visibly $\frac{2\pi}{m!}\mathfrak{F}_{kc_1}(X)$. From Lemma \ref{lemma:Futakic1}, one sees that 
$$ -\mathfrak{F}_{\frac{k^2}{2}c_1^2}(X)= \frac{k^2}{2}\left(\int \mu_X \dfrac{\Ric(\omega)\wedge\Ric(\omega)\wedge\omega^{m-2}}{(m-2)!}+\Delta^{(\omega)} \mu_X \frac{\Ric(\omega)\wedge\omega^{m-1}}{(m-1)!} \right).$$
Finally, the remaining term in $\nu^2$ involving the Cahen-Gutt moment map was identified in \cite{LLF2} to be $\frac{(2\pi)^2}{(m-1)!}\mathfrak{F}_{\frac{1}{12}c_2-\frac{1}{24}c_1^2}(X)$.
\end{proof}

\end{document}